\RequirePackage{fix-cm}
\documentclass[smallextended,envcountsame]{svjour3} 
\smartqed  
\usepackage{graphicx, amssymb}
\usepackage{times,a4wide,mathrsfs}
\usepackage{amsmath}
\usepackage{amsfonts}

\newcommand{\C}{\mathbb{C}}

\newcommand{\QQ}{\mathbb{Q}}
\newcommand{\NN}{\mathbb{N}}
\newcommand{\PP}{\mathbb{P}}

\newcommand{\MM}{\mathcal M}

\newcommand{\LL}{\mathbb L}

\newcommand{\gr}{\hbox{Gr}}
\newcommand{\wt}{\widetilde}
\newcommand{\ima}{\hbox{Im}}
\newcommand{\rom}{\romannumeral}

\makeatletter
\newcommand*{\da@rightarrow}{\mathchar"0\hexnumber@\symAMSa 4B }
\newcommand*{\da@leftarrow}{\mathchar"0\hexnumber@\symAMSa 4C }
\newcommand*{\xdashrightarrow}[2][]{%
  \mathrel{%
    \mathpalette{\da@xarrow{#1}{#2}{}\da@rightarrow{\,}{}}{}%
  }%
}
\newcommand{\xdashleftarrow}[2][]{%
  \mathrel{%
    \mathpalette{\da@xarrow{#1}{#2}\da@leftarrow{}{}{\,}}{}%
  }%
}
\newcommand*{\da@xarrow}[7]{%
  \sbox0{$\ifx#7\scriptstyle\scriptscriptstyle\else\scriptstyle\fi#5#1#6\m@th$}%
  \sbox2{$\ifx#7\scriptstyle\scriptscriptstyle\else\scriptstyle\fi#5#2#6\m@th$}%
  \sbox4{$#7\dabar@\m@th$}%
  \dimen@=\wd0 %
  \ifdim\wd2 >\dimen@
    \dimen@=\wd2 %
  \fi
  \count@=2 %
  \def\da@bars{\dabar@\dabar@}%
  \@whiledim\count@\wd4<\dimen@\do{%
    \advance\count@\@ne
    \expandafter\def\expandafter\da@bars\expandafter{%
      \da@bars
      \dabar@ 
    }%
  }%
  \mathrel{#3}%
  \mathrel{%
    \mathop{\da@bars}\limits
    \ifx\\#1\\%
    \else
      _{\copy0}%
    \fi
    \ifx\\#2\\%
    \else
      ^{\copy2}%
    \fi
  }%
  \mathrel{#4}%
}
\makeatother

\newcommand\undermat[2]{
  \makebox[0pt][l]{$\smash{\underbrace{\phantom{
    \begin{matrix}#2\end{matrix}}}_{\text{$#1$}}}$}#2}

\newtheorem{convention}{Conventions}

\newtheorem{nonumberingp}{Proposition}

\newtheorem{nonumberingt}{Acknowledgements}

 \journalname{}

\begin{document}

\title{Algebraic cycles on Fano varieties of some cubics}

\author{Robert Laterveer}

\institute{CNRS - IRMA, Universit\'e de Strasbourg \at
              7 rue Ren\'e Descartes \\
              67084 Strasbourg cedex\\
              France\\
              \email{laterv@math.unistra.fr}   }

\date{Received: date / Accepted: date}

\maketitle

\begin{abstract} We study cycle--theoretic properties of the Fano variety of lines on a smooth cubic fivefold. The arguments are based on the fact that this Fano variety has finite--dimensional motive. We also present some results concerning Chow groups of Fano varieties of lines on certain cubics in other dimensions.
\end{abstract}



\section{Introduction}

The notion of finite--dimensional motive, developed independently by Kimura and O'Sullivan \cite{Kim}, \cite{An}, \cite{MNP}, \cite{J4}, \cite{Iv} has given major new impetus to the field of algebraic cycles. To give but one example: thanks to this notion, we now know the Bloch conjecture is true for surfaces of geometric genus zero that are rationally dominated by a product of curves \cite{Kim}. 

Examples of varieties known to have finite--dimensional motive are rather scarce (cf. for instance \cite[remark 2]{moi}). 
In this note, we consider some recent members of the club: Fano varieties $F$ of lines on a smooth cubic fivefold $X$. These varieties $F$ are smooth projective of dimension $6$;
the fact that $F$ has finite--dimensional motive was established in \cite{moi}. We deduce some consequences concerning the Chow groups $A^j(F)$ of $F$. Here is a sample of the main results of this note:

\begin{nonumberingp}[=corollary \ref{Chow}] Let $F=F(X)$ be the Fano variety of lines on a smooth cubic fivefold $X\subset\PP^6(\C)$. Let $A^j_{AJ}(F)$ denote the kernel of the Abel--Jacobi map. Then
  \[ A^j_{AJ}(F)_{\QQ}=0\ \ \ \hbox{for\ all\ }j\not=4\ .\]
  \end{nonumberingp}
  
\begin{nonumberingp}[=proposition \ref{propbloch}]  Let $F=F(X)$ be the Fano variety of lines on a smooth cubic fivefold $X\subset\PP^6(\C)$. Let $\Gamma\in A^6(F\times F)$ be a correspondence
such that
  \[ \Gamma_\ast=\hbox{id}\colon\ \ H^{4,2}(F)\ \to\ H^{4,2}(F)\ .\]
  Then
  \[ \Gamma_\ast\colon\ \ A^4_{AJ}(F)_{\QQ}\ \to\ A^4_{AJ}(F)_{\QQ} \]
  is an isomorphism.
 \end{nonumberingp}
 
Other consequences concern Voevodsky's smash--nilpotence conjecture (proposition \ref{propvoe}), and Murre's conjectures (proposition \ref{propmurre}).

We also present some results concerning Fano varieties of lines on cubics of other dimensions. Most of these results are restricted to special cubics $X$ (typically, in order to ensure that $X$  has finite--dimensional motive). Here is a sample of these results:

  \begin{nonumberingp}[=proposition \ref{propvoe6}] Let $X\subset\PP^7(\C)$ be a cubic sixfold defined by an equation
   \[   f_1(x_0,\ldots,x_3)+ f_2(x_4,\ldots,x_7)=0\ ,\]
   where $f_1, f_2$ define smooth cubic surfaces.
  Let $F$ be the Fano variety of lines on $X$ (so $F$ is smooth of dimension $8$). Then numerical and smash--equivalence coincide on $A^j(F)$ for all $j\not=4$.
  \end{nonumberingp}

 \begin{nonumberingp}[=proposition \ref{fermat}] Let $X\subset\PP^{n+1}(\C)$ be the Fermat cubic
     \[ x_0^3+x_1^3+\cdots + x_{n+1}^3=0\ ,\]
   and let $F=F(X)$ be the Fano variety of lines on $X$ (so $F$ is smooth of dimension $2n-4$).
   The cycle class map
   \[ A^j(F)_{\QQ}\ \to\ H^{2j}(F,\QQ) \]
   is injective for all $j\ge {(5n-2)/ 3}$.
 \end{nonumberingp}
  
 \begin{nonumberingp}[=proposition \ref{propDM}] Let $X\subset\PP^{n+1}(\C)$ be a smooth cubic, and let $F=F(X)$ be the Fano variety of lines on $X$ (so $F$ is smooth of dimension $2n-4$). The cycle class map
   \[ A^j(F)_{\QQ}\ \to\ H^{2j}(F,\QQ) \]
   is injective for $j\ge 2n-6$ (resp. $j\ge 2n-7$) provided $n\ge 14$ (resp. $n\ge 18$).
  \end{nonumberingp} 
  
 Proposition \ref{propDM} provides a partial confirmation of a conjecture of Debarre and Manivel \cite[Conjecture 6.4]{DMan} (cf. conjecture \ref{DMconj}).

\vskip0.6cm

\begin{convention} In this note, the word {\sl variety\/} will refer to a reduced irreducible scheme of finite type over $\C$. 

All Chow groups will be with rational coefficients: For any variety $X$, we will denote by $A_j(X)$ the Chow group of $j$--dimensional cycles on $X$ with $\QQ$--coefficients.
For $X$ smooth of dimension $n$ the notations $A_j(X)$ and $A^{n-j}(X)$ will be used interchangeably. 

The notations 
$A^j_{hom}(X)$ and $A^j_{AJ}(X)$ will be used to indicate the subgroups of 
homologically, resp. Abel--Jacobi trivial cycles.
For a morphism $f\colon X\to Y$, we will write $\Gamma_f\in A_\ast(X\times Y)$ for the graph of $f$.
The contravariant category of Chow motives (i.e., pure motives with respect to rational equivalence as in \cite{Sch}, \cite{MNP}) will be denoted $\MM_{\rm rat}$. The category of pure motives with respect to homological (resp. numerical) equivalence will be denoted $\MM_{\rm hom}$ (resp. $\MM_{\rm num}$).

We will write $H^j(X)$ to indicate singular cohomology $H^j(X,\QQ)$.
\end{convention}

\section{Finite--dimensional motives}

We refer to \cite{Kim}, \cite{An}, \cite{MNP}, \cite{Iv}, \cite{J4} for basics on the notion of finite--dimensional motive. 
An essential property of varieties with finite--dimensional motive is embodied by the nilpotence theorem:

\begin{theorem}[Kimura \cite{Kim}]\label{nilp} Let $X$ be a smooth projective variety of dimension $n$ with finite--dimensional motive. Let $\Gamma\in A^n(X\times X)_{}$ be a correspondence which is numerically trivial. Then there is $N\in\NN$ such that
     \[ \Gamma^{\circ N}=0\ \ \ \ \in A^n(X\times X)_{}\ .\]
\end{theorem}

 Actually, the nilpotence property (for all powers of $X$) could serve as an alternative definition of finite--dimensional motive, as shown by a result of Jannsen \cite[Corollary 3.9]{J4}.
   Conjecturally, all smooth projective varieties have finite--dimensional motive \cite{Kim}. 
For examples of varieties known to have finite--dimensional motive, cf. \cite[remark 2]{moi} and the references given there.   

In this note, we will use finite--dimensionality in the guise of the following two lemmas:

 \begin{lemma}[Vial \cite{V3}]\label{split} Let 
  \[ f\colon\ \ P\ \to\ R\ \ \ \hbox{in}\ \MM_{\rm rat}\]
  be a homomorphism of motives, where $P$ is finite--dimensional. Then there exist splittings
  \[  P=P_1\oplus P_2\ ,\ \ \ R=R_1\oplus R_2\ \ \ \hbox{in}\ \MM_{\rm rat}\ \]
  such that $f$ induces an isomorphism $P_1\cong R_1$, and the homomorphism $P_2\to R$ induced by $f$ is numerically trivial.
  \end{lemma}
  
  \begin{proof} This follows from \cite[Lemma 3.6]{V3}.
  For completeness' sake (and because the statement of loc. cit. is slightly different), we include a proof.
  The homomorphism $f$ induces a homomorphism of numerical motives
    \[ \bar{f}\colon\ \ \bar{P}\ \to\ \bar{R}\ \ \ \hbox{in}\ \MM_{\rm num}\ .\]
    Since $\MM_{\rm num}$ is semi--simple \cite{J0}, there exist splittings
    \[  \bar{P}=\bar{P}_1\oplus\bar{P}_2\ ,\ \ \ \bar{R}=\ima \bar{f}\oplus\bar{R}^\prime\ \ \ \hbox{in}\ \MM_{\rm num}\ \]
    such that $\bar{f}$ induces an isomorphism $\bar{P}_1\cong \ima \bar{f}$, and $\bar{P}_2\to \bar{R}$ and $\bar{P}\to \bar{R}^\prime$ are zero.
    Using finite--dimensionality, $\bar{P}_1$ lifts to a direct summand $P_1$ of $P\in\MM_{\rm rat}$. This induces a splitting $P=P_1\oplus P_2$, where the class of $P_2$ in $\MM_{\rm num}$ coincides with the afore--mentioned $\bar{P}_2$ (and so the homomorphism $P_2\to R$ induced by $f$ is numerically trivial).
     
    Let $s\colon P_1\to R$ denote the restriction of $f$ to $P_1$. Let $\bar{t}\colon \bar{R}\to\bar{P}_1$ be a left--inverse to $\bar{s}\colon \bar{P}_1\to \bar{R}$.
    By finite--dimensionality of $P_1$, there exists $t\colon R\to P_1$ which is left--inverse to $s\colon P_1\to R$ (lemma \ref{ok}). The idempotent $s\circ t\colon R\to R$
    defines a direct summand $R_1\subset R$ such that $P_1\cong R_1$. This induces a splitting $R=R_1\oplus R_2$, 
    where $R_2$ is defined by the idempotent $\hbox{id}_R-s\circ t$.
      \end{proof}  
  
\begin{lemma}[Vial \cite{V3}]\label{ok} Let  \[ f\colon\ \ P\ \to\ R\ \ \ \hbox{in}\ \MM_{\rm rat}\]
  be a homomorphism of motives, where $P$ is finite--dimensional. Assume 
  \[ \bar{f}\colon\ \ \bar{P}\ \to\ \bar{R}\ \ \ \hbox{in}\ \MM_{\rm num}\ \]
  has a left--inverse. Then also $f$ has a left--inverse.
  \end{lemma}
  
  \begin{proof} This is a remark in \cite[Section 3.3]{V3}.
   \end{proof}

\section{A relation between Chow motives}

We recall the main result of \cite{moi}:

\begin{theorem}[\cite{moi}]\label{main} Let $X\subset\PP^{n+1}(\C)$ be a smooth cubic hypersurface. Let $F:=F(X)$ denote the Fano variety of lines on $X$, and let $X^{[2]}$ denote the second Hilbert scheme of $X$. There is an isomorphism of Chow motives
  \[ h(F)(2)\oplus\bigoplus_{i=0}^n h(X)(j)\cong h(X^{[2]})\ \ \ \hbox{in}\ \MM_{\rm rat}\ .\]
\end{theorem}

\begin{corollary}\label{corfin} Let $X$ and $F$ be as in theorem \ref{main}. If $X$ has finite--dimensional motive, then also $F$ has finite--dimensional motive.
\end{corollary}

\begin{corollary}\label{lef} Let $X$ and $F$ be as in theorem \ref{main}. Then $F$ satisfies the standard conjecture of Lefschetz type.
\end{corollary}

\begin{proof} Theorem \ref{main} implies there is the same relation in $\MM_{\rm hom}$. This means that $F$ is motivated by $X$, in the sense of \cite{A}. The validity of the standard conjecture of Lefschetz type for $F$ then follows from its validity for $X$, by virtue of \cite[Lemma 4.2]{A}.
\end{proof}

\begin{remark} In case $X$ is a cubic of dimension $n=4$, the Fano variety $F=F(X)$ is a hyperk\"ahler fourfold of $K3^{[2]}$ type \cite{BD}. As such, the standard conjecture of Lefschetz type for $F$ also follows from \cite[Theorem 1.1]{ChMa}.
\end{remark}

\section{Fano varieties of cubic fivefolds}

\subsection{The motive of $F$}

\begin{theorem}\label{main5} Let $F$ be the Fano variety of lines on a smooth cubic fivefold. There is a curve $C\subset F$ and a homomorphism
  \[ h(F)\ \to\ h^2(C\times C)(2)\oplus \bigoplus_{i=0}^5 \bigl( h^1(C)\oplus h^1(C)\bigr)(i)\oplus \bigoplus_{i=1}^4 h^1(C)(i)\oplus \bigoplus \LL(n_j)\ \ \ \hbox{in}\ \MM_{\rm rat}\ \]
  admitting a left--inverse
  (i.e., the motive $h(F)$ can be identified with a direct summand of the right--hand side).
  \end{theorem}
  
  \begin{proof} Let $X\subset\PP^6(\C)$ be a smooth cubic fivefold, and let $F=F(X)$ be the Fano variety of lines on $X$. Lewis' cylinder homomorphism \cite{Lew} furnishes a certain complete intersection curve $C\subset F$, and a correspondence $\Psi\in A^{3}(X\times C)$ such that the composition
    \[  H^5(X)\ \xrightarrow{\Psi_\ast}\ H^1(C)\ \xrightarrow{\Psi^\ast}\ H^5(X) \]
    is a multiple of the identity.
  Since the other cohomology groups $H^i(X)$ for $i\not= 5$ are algebraic, it follows that the homomorphism of motives 
   \[ \Psi\colon\ \ h(X)\ \to\  h^1(C)(2)\oplus \bigoplus_{i=0}^5 \LL(i)\ \ \ \hbox{in}\ \MM_{\rm hom}\ \]
   admits a left--inverse.
  As is well--known, 
    \[ A^\ast_{AJ}(X)=0 \]
  \cite{Lew} (cf. \cite{Otw} or \cite{HI} for alternative proofs). As such, $X$ has finite--dimensional motive \cite[Theorem 4]{V2}. Using lemma \ref{ok}, this implies the homomorphism of Chow motives
  \[  \Psi\colon\ \ h(X)\ \to\  h^1(C)(2)\oplus \bigoplus_{i=0}^5 \LL(i)\ \ \ \hbox{in}\ \MM_{\rm rat}\ \] 
  admits a left--inverse.
  
  Theorem \ref{main} implies there is a homomorphism
   \[ h(F)\ \to\ h(X^{[2]})(-2)\ \ \ \hbox{in}\ \MM_{\rm rat}\ \]
   admitting a left--inverse. Since $X^{[2]}$ is dominated by the blow--up of $X\times X$ along the diagonal, there is also a homomorphism
   \[ h(X^{[2]})\ \to\ h(X\times X)\oplus \bigoplus_{i=1}^4 h(X)(i)\ \ \ \hbox{in}\ \MM_{\rm rat}\ \]
   admitting a left--inverse. Composing, we find a homomorphism
   \[  \Gamma\colon\ \ h(F)\ \to\ h(X\times X)(-2)\oplus \bigoplus_{i=1}^4 h(X)(i-2)\ \ \ \hbox{in}\ \MM_{\rm rat} \]
    admitting a left--inverse.   
    
   The composition
   \[ \begin{split} h(F)\ \xrightarrow{\Gamma}\  &h(X\times X)(-2)\oplus \bigoplus_{i=1}^4 h(X)(i-2)\ \xrightarrow{(\Psi\times\Psi,\Psi,\Psi,\Psi,\Psi)}\\
       &\Bigl(h^1(C)(2)\oplus \bigoplus_{i=0}^5 \LL(i)\Bigr)^{\otimes 2}(-2) \oplus \bigoplus_{i=1}^4 \Bigl(  h^1(C)(2)\oplus \bigoplus_{j=0}^5 \LL(j)  \Bigr) (i-2)\ \to\ 
        \\
        & h^2(C\times C)(2)\oplus \bigoplus_{i=0}^5 \bigl( h^1(C)\oplus h^1(C)\bigr)(i)\oplus \bigoplus_{i=1}^4 h^1(C)(i)\oplus \bigoplus \LL(n_j)\ \ \ \hbox{in}\ \MM_{\rm rat}\\
         \end{split} \]
   has a left--inverse (the last arrow is just regrouping and identifying $h^1(C)\otimes h^1(C)$ with a direct summand of $h^2(C\times C)$).      
        \end{proof} 


\begin{corollary}\label{Chow} Let $F$ be the Fano variety of lines on a smooth cubic fivefold. Then 
   \[  \begin{split}  
                      & A^j_{AJ}(F)=0\ \ \ \hbox{for\ all\ }j\not= 4\ .\\
                      \end{split}\]
            \end{corollary} 
            
      \begin{proof} This is immediate from theorem \ref{main5}: indeed, there is a split injection
      \[ A^j_{AJ}(F)\ \to\ A^{j-2}_{AJ}(C\times C)\ ,\]
      and the right--hand side is zero for $j\not=4$.
      \end{proof}                
            
 \begin{corollary} Let $F$ be the Fano variety of lines on a smooth cubic fivefold. The generalized Hodge conjecture is true for $F$.
 \end{corollary}
                 
\begin{proof} Corollary \ref{Chow} implies
   \[ \hbox{Niveau}(A^\ast F)\le 2\ ,\]
   in the language of \cite{small}. A Bloch--Srinivas argument (for instance as in loc. cit.) then implies the corollary.
  \end{proof}

\subsection{A refined Chow--K\"unneth decomposition}

\begin{definition} Let $M$ be a smooth projective variety of dimension $m$. We say that $M$ has a Chow--K\"unneth decomposition if there exists a decomposition of the diagonal
  \[ \Delta_M= \Pi_0+ \Pi_1+\cdots +\Pi_{2m}\ \ \ \hbox{in}\ A^m(M\times M)\ ,\]
  such that the $\Pi_i$ are mutually orthogonal idempotents and $(\Pi_i)_\ast H^\ast(M)= H^i(M)$.
\end{definition}

\begin{remark} The existence of a Chow--K\"unneth decomposition for any smooth projective variety is part of Murre's conjectures (cf. \cite{Mur}\, \cite{J} and section \ref{secmur} below).
\end{remark}

\begin{proposition}\label{CK} Let $F$ be the Fano variety of lines on a smooth cubic fivefold. There exists a Chow--K\"unneth decomposition 
  \[ \Delta_F= \Pi_0+\Pi_1 +\cdots +\Pi_{12}\ \ \ \hbox{in}\ A^6(F\times F) \]
 (i.e., the $\Pi_i$ are mutually orthogonal idempotents summing to the diagonal). Moreover, 
    one may assume $\Pi_i$ factors over a curve for $i$ odd (i.e., for any odd $i$ there exists a curve $C_i$ and correspondences $\Gamma_i\in A^{13-i\over 2}(C_i\times F)$, $\Psi_i\in A^{i+1\over 2}(F\times C_i)$ such that $\Pi_i=\Gamma_i\circ \Psi_i$ in $A^6(F\times F)$), and $\Pi_{2i}$ factors over a $0$--dimensional variety for $i\not=3$, and $\Pi_6$ factors over a surface.  
       
  There is a further splitting in orthogonal idempotents
     \[ \Pi_6 = \Pi_6^{tr} + \Pi_6^{alg}\ \ \ \hbox{in}\ A^6(F\times F)  \ ,\]    
     such that
     \[ \begin{split}  &(\Pi_6^{tr})_\ast H^\ast(F,\QQ)=  H^6_{tr}(F)\ ,  \\
                        &(\Pi_6^{tr})_\ast A^\ast(F)= A^4_{AJ}(F)\ .\\
                        \end{split}\]
       Here, $H^6_{tr}(F)$ denotes the orthogonal complement to $N^3 H^6(F)$ under the cup--product pairing.                 
               \end{proposition}          

\begin{proof} This follows from work of Vial \cite{V4}. Indeed, using theorem \ref{main5} one sees that the cohomology of $F$ is spanned (via the action of correspondences) by the cohomology of a surface. This implies that $F$ verifies condition (*) of \cite{V4}. Since $F$ has finite--dimensional motive, $F$ also verifies condition (**) of loc. cit. Thus, \cite[Theorems 1 and 2]{V4} apply, which means there is a refined Chow--K\"unneth decomposition $\Pi_{i,j}$ such that $\Pi_{i,j}$ acts on cohomology as a projector onto $\gr^j_{N} H^i(F)$, where $N^\ast$ is the coniveau filtration (here we use that Vial's niveau filtration $\wt{N}^\ast$ coincides with $N^\ast$ on $H^\ast(F)$, since $N^\ast=\wt{N}^\ast$ for surfaces).

We now remark that
  \[  \Pi_i= {\displaystyle\sum_j} \Pi_{i,j}= \Pi_{i,\lceil {i-1\over 2}\rceil}\ \ \ \hbox{for\ all\ }i\not=6\ \]
 (this follows from \cite[Theorem 2]{V4} since $H^i(F)= N^{\lceil {i-1\over 2}\rceil} H^i(F)$ for all $i\not=6$).
 
 For $i=6$, we define
   \[  \Pi_6^{tr}:=\Pi_{6,2}\ ,\ \ \ \Pi_6^{alg}:= \Pi_{6,3}\ \ \ \in A^6(F\times F)\ .\]
   It follows from \cite[Theorems 1 and 2]{V4} these Chow--K\"unneth projectors have the stated properties.
 \end{proof}

\begin{remark} The Chow--K\"unneth projectors are not uniquely defined. Yet, the motive
  \[ t_6(F):=(F,\Pi_6^{tr},0)\ \ \in \MM_{\rm rat} \]
  is well--defined up to isomorphism; this follows from \cite[Theorem 7.7.3]{KMP} and \cite[Proposition 1.8]{V4}.
The motive $t_6(F)$ can be seen as an analogue of the ``transcendental part'' of the motive of a surface $t_2(S)$ constructed for any surface $S$ in \cite{KMP}.
\end{remark}

\subsection{Bloch's conjecture}

\begin{proposition}\label{propbloch} Let $F$ be the Fano variety of lines on a smooth cubic fivefold, and let $\Gamma\in A^6(F\times F)$ be a correspondence.

\noindent
(\rom1) Assume 
  \[ \Gamma_\ast=\hbox{id}\colon\ \ H^{4,2}(F)\ \to\ H^{4,2}(F)\ .\]
  Then
  \[ \Gamma_\ast\colon\ \ A^4_{AJ}(F)\ \to\ A^4_{AJ}(F) \]
  is an isomorphism.
  
\noindent
(\rom2) Assume
\[   \Gamma_\ast\colon\ \ H^{4,2}(F)\ \to\ H^{4,2}(F)\ \]
is $0$.
  Then there exists $N$ such that
  \[ (\Gamma^{\circ N})_\ast\colon\ \ A^4_{AJ}(F)\ \to\ A^4_{AJ}(F) \]
  is $0$.
\end{proposition}

\begin{proof} 

\noindent
(\rom1) Let $\Pi_i$ and $\Pi_6=\Pi_6^{tr}+\Pi_6^{alg}$ be a refined Chow--K\"unneth decomposition as in proposition \ref{CK}.  By assumption, we have
   \[  \Gamma\circ\Pi_6^{tr}-\Pi_6^{tr}=0\ \ \ \hbox{in}\ H^{12}(F\times F)\ .\]
   It follows there is $N\in\NN$ such that
   \[ \bigl(\Gamma\circ\Pi_6^{tr}-\Pi_6^{tr}\bigr)^{\circ N}=0\ \ \ \hbox{in}\ A^{6}(F\times F)\ .\]
   Developing, this implies
   \[ (\Gamma\circ\Pi_6^{tr})^{\circ N} + \cdots +(-1)^N \Pi_6^{tr}=0\ \ \ \hbox{in}\ A^6(F\times F)\ ,\]
   where the dots indicate compositions of $\Gamma$ and $\Pi_6^{tr}$, i.e. elements of type
     \[  \pm\, \Gamma^{\circ r_1}\circ\Pi_6^{tr}\circ \Gamma^{\circ r_2}\circ \Pi_6^{tr}\circ \cdots \circ \Gamma^{\circ r_\ell}\circ \Pi_6^{tr} \ .\]
      Considering the action on $A^4_{AJ}(F)$ (and using that $\Pi_6^{tr}$ acts as the identity on $A^4_{AJ}(F)$), this implies
   \[   \Bigl(  \Gamma^{\circ N}+ c_{N-1} \Gamma^{\circ N-1}+\cdots + c_1 \Gamma\Bigr){}_\ast=\hbox{id}\colon\ \ \ A^4_{AJ}(F)\ \to\ A^4_{AJ}(F)\ ,\]
   where $c_j\in \NN$. This proves (\rom1), as we have constructed an inverse to $\Gamma_\ast$:
   \[ \begin{split} \Gamma_\ast &\Bigl( \Gamma^{\circ N-1} + c_{N-1} \Gamma^{\circ N-2} +\cdots +c_1\Bigr){}_\ast = \\  
       &\Bigl( \Gamma^{\circ N-1} + c_{N-1} \Gamma^{\circ N-2} +\cdots +c_1\Bigr){}_\ast \Gamma_\ast=\hbox{id}\colon\ \ \ A^4_{AJ}(F)\ \to\ A^4_{AJ}(F)\ .\\
       \end{split}\]
  
 \noindent
 (\rom2) By assumption, we have
    \[  \Gamma\circ\Pi_6^{tr}=0\ \ \ \hbox{in}\ H^{12}(F\times F)\ .\]             
  It follows there exists $N\in\NN$ such that
    \[ (\Gamma\circ\Pi_6^{tr})^{\circ N}=0\ \ \ \hbox{in}\ A^6(F\times F)\ .\]
    This implies
    \[   \bigl((\Gamma\circ\Pi_6^{tr})^{\circ N}\bigr){}_\ast=(\Gamma^{\circ N})_\ast=0\colon\ \ \     A^4_{AJ}(F)\ \to\ A^4_{AJ}(F)\ .\]    
\end{proof}

\begin{remark} Proposition \ref{propbloch} establishes a weak version of Bloch's conjecture. Indeed, it is expected that in (\rom1) one actually has that $\Gamma_\ast$ is the identity, and that in (\rom2) one has $N=1$.
\end{remark}

\subsection{Voevodsky's conjecture}

 \begin{definition}[Voevodsky \cite{Voe}]\label{sm} Let $Y$ be a smooth projective variety. A cycle $a\in A^j(Y)$ is called {\em smash--nilpotent\/} 
if there exists $m\in\NN$ such that
  \[ \begin{array}[c]{ccc}  a^m:= &\undermat{(m\hbox{ times})}{a\times\cdots\times a}&=0\ \ \hbox{in}\  A^{mj}(Y\times\cdots\times Y)_{}\ .
  \end{array}\]
  \vskip0.6cm

Two cycles $a,a^\prime$ are called {\em smash--equivalent\/} if their difference $a-a^\prime$ is smash--nilpotent. 
We will write 
    \[A^j_\otimes(Y)\subset A^j(Y)\] 
  for the subgroup of smash--nilpotent cycles.
\end{definition}

\begin{conjecture}[Voevodsky \cite{Voe}]\label{voe} Let $Y$ be a smooth projective variety. Then
  \[  A^j_{num}(Y)\ \subset\ A^j_\otimes(Y)\ \ \ \hbox{for\ all\ }j\ .\]
  \end{conjecture}

\begin{remark} It is known that $A^j_{alg}(Y)\subset A^j_\otimes(Y)$ for any $Y$ \cite{Voe}, \cite{V9}. In particular, conjecture \ref{voe} is true for divisors ($j=1$) and for $0$--cycles ($j=\dim Y$).

It is known \cite[Th\'eor\`eme 3.33]{An} that conjecture \ref{voe} for all smooth projective varieties implies (and is strictly stronger than) Kimura's finite--dimensionality conjecture for all smooth projective varieties. For partial results concerning conjecture \ref{voe}, cf. \cite{KS}, \cite{Seb2}, \cite{Seb}, \cite[Theorem 3.17]{V3}, \cite{moismash}.
\end{remark}

\begin{proposition}[Sebastian \cite{Seb}]\label{seb} Let $Y$ be a smooth projective variety of dimension $d$, dominated by a product of curves. Then
  \[ A^{d-1}_{num}(Y)=A^{d-1}_\otimes(Y)\ .\]
  \end{proposition}
  
  \begin{proof} This is \cite[Theorem 6]{Seb}. 
  \end{proof}

We now prove the main result of this subsection:

\begin{proposition}\label{propvoe} Let $F$ be the Fano variety of lines on a smooth cubic fivefold.

\noindent
(\rom1) \[ A^j_{num}(F)=A^j_\otimes(F)\ \ \ \hbox{for\ all\ }j\ .\]

\noindent
(\rom2) \[ A^j_{num}(F\times F)=A^j_\otimes(F\times F)\ \ \ \hbox{for\ all\ }j\not=6\ .\]

\noindent
(\rom3) \[ A^j_{num}(F^3)=A^j_\otimes(F^3)\ \ \ \hbox{for\ }j\le 3\ \hbox{and\ for\ }j\ge 15\ .\]

\end{proposition}

\begin{proof}
Part (\rom1) is immediate: $A^\ast(F)$ factors over the Chow groups of a surface, and conjecture \ref{voe} is known for surfaces.

\noindent(\rom2)
As we have seen, there is an inclusion
  \[ h(F)\ \subset\ h^2(C\times C)(2)\oplus \bigoplus_{i=0}^5 \bigl( h^1(C)\oplus h^1(C)\bigr)(i)\oplus \bigoplus_{i=1}^4 h^1(C)(i)\oplus\bigoplus_{n_j\ge 0} \LL(n_j)\ \ \ \hbox{in}\ \MM_{\rm rat}\ .\]
  
  It follows there is also an inclusion
  \[ h(F\times F)\ \subset\ h^4(C^4)(4)\oplus \bigoplus h^2(C\times C)(m_i)\oplus \bigoplus \LL(n_i)\oplus M_{odd}\ \ \ \hbox{in}\ \MM_{\rm rat}\ ,\]
  where $M_{odd}$ is an oddly finite--dimensional motive, in the sense of \cite{Kim}.
  This implies there is a correspondence $\Gamma$ giving a split injection of Chow groups
  \[ \Gamma_\ast\colon\ \ A^j_{num}(F\times F)\ \to\ A^{j-4}_{num}(C^4)\oplus \bigoplus A^{j-m_i}_{num}(C\times C)\oplus A^\ast(M_{odd})\ .\]
 The left--inverse to $\Gamma_\ast$ is again induced by a correspondence. The group $A^\ast_{num}(C\times C)$ consists of smash--nilpotent cycles (this is true for any surface). The group $A^\ast(M_{odd})$ consists of smash--nilpotent cycles (lemma \ref{odd} below). The group $A^{j-4}_{num}(C^4)$ consists of smash--nilpotent cycles for $j\not= 6$, in view of proposition \ref{seb}. This implies that
   \[ \Gamma_\ast A^j_{num}(F\times F)\ \subset\ A^\ast_\otimes ()\ \ \ \hbox{for\ all\ }j\not= 6\ .\]
   Since the left--inverse to $\Gamma_\ast$ is induced by a correspondence, it preserves smash--nilpotence and thus
   \[    A^j_{num}(F\times F)\ \subset\ A^\ast_\otimes (F\times F)\ \ \ \hbox{for\ all\ }j\not= 6\ .\]
    
    \begin{lemma}[Kimura \cite{Kim}]\label{odd} Suppose $M\in\MM_{\rm rat}$ is oddly finite--dimensional. Then
       \[ A^\ast_{}(M)\ \subset\ A^\ast_\otimes(M)\ .\]
     \end{lemma}
   
   \begin{proof} This is \cite[Proposition 6.1]{Kim}.
   \end{proof}

\noindent
(\rom3)
As above, we find there is an inclusion of motives
  \[ h(F^3)\ \subset\ h^6(C^6)(6)\oplus \bigoplus_{2\le m_i\le 12} h^4(C^4)(m_i)\oplus \bigoplus h^2(C^2)(n_i)\oplus \bigoplus \LL(p_i)\oplus M_{odd}\ \ \ \hbox{in}\ \MM_{\rm rat}\ ,\]
  with $M_{odd}$ oddly finite--dimensional.
  This implies there is a correspondence $\Gamma$ inducing a split injection of Chow groups
   \[ \Gamma_\ast\colon\ \ A^j_{num}(F^3)\ \to\ A^{j-6}_{num}(C^6)\oplus \bigoplus_{2\le m_i\le 12} A^{j-m_i}_{num}(C^4)\oplus \bigoplus A^\ast_{num}(C^2)\oplus A^\ast(M_{odd})\ .\]
  As above, the last two terms consist of smash--nilpotent cycles. The first and second term are smash--nilpotent provided $j\le 3$ or $j\ge 15$ (for $j=15$ we use again proposition \ref{seb}). We conclude using the fact that the left--inverse to $\Gamma_\ast$ is given by a correspondence.

\end{proof}

\subsection{Murre's conjectures}
\label{secmur}

\begin{conjecture}[Murre \cite{Mur}]\label{conjmurre} Let $X$ be a smooth projective variety of dimension $n$.

\noindent
(A) There exists a Chow--K\"unneth decomposition for $X$, i.e. a mutually orthogonal set of idempotents $\Pi_i\in A^n(X\times X)$ summing to the diagonal and such that $(\Pi_i)_\ast H^\ast(X)=H^i(X)$ for all $i$.

\noindent
(B) $(\Pi_i)_\ast A^j(X)=0$ for $i<j$ and for $i>2j$.

\noindent
(C) The filtration
   \[  F^r A^j(X):= \ker (\Pi_{2j})\cap \ker (\Pi_{2j-1})\cap\cdots\cap \ker (\Pi_{2j-r+1}) \ \ \subset\ A^j(X)\]
   is independent of the choice of the $\Pi_i$.
   
 \noindent
 (D) $F^1 A^j(X)=A^j_{hom}(X)$.
 \end{conjecture}

\begin{proposition}\label{propmurre} Let $F$ be the Fano variety of lines on a smooth cubic fivefold.

\noindent
(\rom1) Murre's conjectures (A), (B), (C) and (D) are true for $F$.

\noindent
(\rom2) Murre's conjectures (A) and (B) are true for $F\times F$.
\end{proposition}

\begin{proof} It will be convenient to extend Murre's conjectures to motives; this has been done in \cite[Section 2]{V3}. We will also rely on the following result:

\begin{proposition}[Murre, Xu and Xu, Vial]\label{le4} 
Murre's conjectures (B) and (D) are true for the product of at most $3$ curves. Murre's conjecture (B) is true for the product of $4$ curves.
\end{proposition}

\begin{proof} The first statement is proven in \cite[Part II]{Mur}. The second statement follows from the first by applying \cite[Theorem 3.3]{XuXu} (or alternatively \cite[Theorem 2.6]{V3}).
\end{proof}

We first prove (\rom2) of proposition \ref{propmurre}. As we have seen, there is an inclusion of motives
  \[  h(F\times F)\ \subset\ h(C^4)(4) \oplus \bigoplus_{r\le 3} h(C^r)(m_j)\ \ \ \hbox{in}\ \MM_{\rm rat}\ .\]
The right--hand side verifies conjecture (B) (proposition \ref{le4}). Since any direct summand of a motive verifying (B) also verifies (B) \cite[Proposition 2.7]{V3}, it follows that $F\times F$ verifies (B).

We now prove proposition \ref{propmurre}(\rom1). That (B) and (D) hold for $F$ follows as above. It remains to prove (C); this follows from (\rom2) by applying \cite[Proposition 2.8]{V3}. 
\end{proof}

\section{Some cubic fourfolds}

In this section, we consider Voevodsky's conjecture (conjecture \ref{voe}) for Fano varieties of some special cubic fourfolds (known to have finite--dimensional motive).

 \begin{proposition}\label{propvoe4} Let $X$ be a cubic fourfold defined by an equation
   \[   f(x_0,\ldots,x_3)+x_4^3+x_5^3=0\ ,\]
   where $f(x_0,\ldots,x_3)$ defines a smooth cubic surface.
  Let $F$ be the Fano variety of lines on $X$ (so $F$ is smooth of dimension $4$). Then
    \[ A^3_{num}(F)= A^3_{\otimes}(F)\ .\]
  \end{proposition}
  
  \begin{proof} We know $X$ and $F$ have finite--dimensional motive \cite[corollary 12(\rom3)]{moi}.  
   To prove the proposition, we reduce to abelian varieties, where it is known that smash--equivalence and numerical equivalence coincide for $1$--cycles. For an abelian variety $A$, we will write $\Pi_j^A$ for the Chow--K\"unneth decomposition \cite{Sch}, \cite{DM}. We will write $h^j(A)$ for the motive $(A,\Pi_j^A,0)$.
  
  We first prove the following statement, which may be of independent interest:
   
   \begin{proposition}\label{ab} Let $F$ be as in proposition \ref{propvoe4}. There exists an abelian variety $A$ (of dimension $22$) such that
     \[ h(F)\ \subset\ h^4(A\times A)\oplus \bigoplus_j h^2(A)(m_j)\oplus \bigoplus_j \LL(n_j)\ \ \ \hbox{in}\ \MM_{\rm rat} \]
     (i.e., $h(F)$ is isomorphic to a direct summand of the right--hand side).
     \end{proposition}
     
   \begin{proof} For a cubic fourfold $X$ as in proposition \ref{propvoe4}, van Geemen and Izadi \cite[Corollary 5.3]{GI} have constructed a Kuga--Satake correspondence, i.e. a correspondence $\Gamma_{KS}$ inducing an injection
      \[  (\Gamma_{KS})_\ast\colon\ \ H^4(X,\QQ)_{prim}\ \to\ H^5(Z,\QQ)\otimes H^1(E,\QQ)\ ,\]
      where $Z$ is a certain smooth cubic fivefold and $E$ is an elliptic curve.
      
      Smooth cubic fivefolds $Z$ are known to have $A^\ast_{AJ}(Z)=0$ (\cite{Lew} or \cite{Otw} or \cite{HI}), whence $H^5(Z,\QQ)=N^2 H^5(Z,\QQ)$ (where $N^\ast$ is the coniveau filtration).  
          Now, \cite[Theorem 1]{ACV} furnishes an abelian variety $J$ (of dimension $h^{2,3}(Z)=21$ ) and a correspondence $\Lambda^\prime$ on $J\times Z$ inducing an isomorphism
  \[  (\Lambda^\prime)_\ast\colon\ \ H^1(J)\ \xrightarrow{\cong}\ H^5(Z)\ .\]
  The correspondence $\Lambda^\prime$ induces an isomorphism 
  \[  \Lambda^\prime\colon\ \ h^1(J)\ \xrightarrow{\cong}\ h^5(Z)\ \ \hbox{in}\ \MM_{\rm num}\ ,\]
  hence there also exists a correspondence $\Lambda$ on $Z\times J$ inducing the inverse isomorphism
  \[ \Lambda \colon\ \ h^5(Z)\ \xrightarrow{\cong}\ h^1(J)\ \ \hbox{in}\ \MM_{\rm num}\ .\]
  Composing correspondences, one obtains an injection
  \[  \bigl((\Lambda\times\Delta_E)\circ \Gamma_{KS}\bigr){}_\ast\colon\ \ 
       H^4(X,\QQ)_{prim}\ \to\ H^2(J\times E,\QQ)\ .\]
   We now write $A:=J\times E$.   
 Since $H^\ast(X,\QQ)$ consists of $H^4(X,\QQ)_{prim}$ plus algebraic cohomology, the correspondence $(\Lambda\times\Delta_E)\circ \Gamma_{KS}$
 induces a homomorphism of motives
  \[ f\colon\ \ h(X)\ \to\ h^2(A)(1)\oplus \bigoplus_j \LL(m_j)\ \ \ \hbox{in}\ \MM_{\rm rat}\ .\]
 Applying lemma \ref{split} to $f$, we find a splitting 
   \[ h(X)= h(X)_1 \oplus h(X)_2 \ \ \ \hbox{in}\ \MM_{\rm rat}\ ,\]
   such that $f$ restricted to $h(X)_2$ is numerically trivial, and $f$ restricted to $h(X)_1$ has a left--inverse.
 Using the fact that $X$ and $A$ verify the Lefschetz standard conjecture (and so homological and numerical equivalence coincide on $X\times A$), we find that $f$ restricted to $h(X)_2$ is homologically trivial. On the other hand, $f_\ast$ is injective on cohomology and so $H^\ast(h(X)_2)=0$. By finite--dimensionality, $h(X)_2=0$ in $\MM_{\rm rat}$
 and so $f$ 
 has a left--inverse, i.e. $h(X)$ is isomorphic to a direct summand of the right--hand side.

  Composing with the result obtained in theorem \ref{main}, this implies
  \[ h(F)\ \subset\   h(X\times X) (-2)\oplus \bigoplus_{j} h(X)(j)\ \subset\  h^4(A\times A)\oplus \bigoplus_j h^2(A)(m_j) \oplus \bigoplus_j \LL(n_j) \ \ \ \hbox{in}\ \MM_{\rm rat}\ ,\]
  and so proposition \ref{ab} is proven.
   \end{proof}

  We now wrap up the proof of proposition \ref{propvoe}. It follows from proposition \ref{ab} (after taking Chow groups) there is a correspondence $\Gamma$ inducing a split injection
    \[ \Gamma_\ast\colon\ \ A^3_{num}(F)\ \to\ (\Pi_4^{A\times A})_\ast A^3(A\times A)\oplus \bigoplus_j (\Pi_2^{A})_\ast A^{n_j}_{num}(A)\ ,\] 
   and the left inverse to $\Gamma_\ast$ is again induced by a correspondence. 
    This implies there is a split injection
    \[ \Gamma_\ast\colon\ \ A^3_{num}(F)\ \to\  A^3_{(2)}(A\times A)\oplus \bigoplus_j  A^{2}_{(2)}(A)\ ,\] 
    where $A^i_{(j)}$ denotes the Beauville filtration \cite{Beau}. (Indeed, one has $(\Pi_2^{A})_\ast A^{1}_{num}(A)=0$ and     $(\Pi_2^{A})_\ast A^{i}_{}(A)=A^i_{(2i-2)}(A)=0$ for $i\ge 3$ \cite{Beau}.)    
   Composing with some Lefschetz operators, we obtain split injections
   \[  (\Gamma^\prime)_\ast\colon\ \ A^3_{num}(F)\ \to\  A^3_{(2)}(A\times A)\oplus \bigoplus_j  A^{2}_{(2)}(A)\ \xrightarrow{\cong}\  A^{2g-1}_{(2)}(A\times A)\oplus \bigoplus_j A^{g}_{(2)}(A)   \ ,\]   
   where $g$ denotes the dimension of $A$. The second arrow is an isomorphism thanks to K\"unnemann's hard Lefschetz theorem \cite{Kun}.  
   
   As Voevodsky's conjecture is true for $0$--cycles and for $1$--cycles on abelian varieties \cite{Seb2}, we know that
   \[ (\Gamma^\prime)_\ast A^3_{num}(F) \]
   consists of smash--nilpotent cycles. Since the left--inverse to $(\Gamma^\prime)_\ast$ is induced by a correspondence, it follows that $A^3_{num}(F)$ consists of smash--nilpotent cycles.
     \end{proof}

   \begin{remark}\label{beauvilleconj} Let $X$ and $F$ be as in proposition \ref{propvoe4}. The above argument does {\em not\/} allow to prove that $A^2_{num}(F)=A^2_\otimes(F)$. The problem is as follows: to prove this along the lines of the above proof, one would need to know that 
    \[A^2_{num}(A\times A)\cap A^2_{(0)}(A\times A)=0\ ,\] 
    which is one of the open cases of Beauville's conjectures.
   \end{remark}

\section{Some cubic sixfolds}

In this sction, we consider Voevodsky's conjecture (conjecture \ref{voe}) for Fano varieties of some special cubic sixfolds (known to have finite--dimensional motive).

\begin{proposition}\label{propvoe6} Let $X$ be a cubic sixfold defined by an equation
   \[   f_1(x_0,\ldots,x_3)+ f_2(x_4,\ldots,x_7)=0\ ,\]
   where $f_1, f_2$ define smooth cubic surfaces.
  Let $F$ be the Fano variety of lines on $X$ (so $F$ is smooth of dimension $8$). Then
    \[ A^j_{num}(F)= A^j_{\otimes}(F)\ \ \ \hbox{for\ all\ }j\not=4\ .\]
  \end{proposition}

\begin{proof} The argument is similar to proposition \ref{propvoe4}. Again, we first prove a ``Kuga--Satake type'' statement that may be of independent interest:

  \begin{proposition}\label{ab6} Let $F$ be as in proposition \ref{propvoe6}. There exists an abelian variety $A$ (of dimension $10$) such that
     \[ h(F)\ \subset\ h^4(A\times A)(2)\oplus \bigoplus_j h^2(A)(m_j)\oplus \bigoplus_j \LL(n_j)\ \ \ \hbox{in}\ \MM_{\rm rat} \]
     (i.e., $h(F)$ is isomorphic to a direct summand of the right--hand side).
     \end{proposition}

\begin{proof} This is similar to proposition \ref{ab}. It follows from the Katsura--Shioda argument \cite[Remark 1.10]{KatS} that there is a rational map
  \[ \phi\colon\ \ X_1\times X_2\ \dashrightarrow\ X\ ,\]
  where $X_1$ and $X_2$ are smooth cubic threefolds defined by
  \[  f_1(x_0,\ldots,x_3)+z^3=0\ ,\]
  resp.
  \[ f_2(x_4,\ldots,x_7)+w^3=0\ .\]
  Moreover, the indeterminacy of the rational map $\phi$ is resolved by the blow--up $\wt{X_1\times X_2}$ with centre $Y_1\times Y_2\subset X_1\times X_2$, where $Y_1, Y_2$ are smooth cubic surfaces.
  Since $H^3(X_i)=N^1 H^3(X_i)$, we have (just as in the proof of proposition \ref{ab}) a split injection 
    \[ h(X_i)\ \to\ h^1(A_i)(1)\oplus \bigoplus_j \LL(n_j)\ \ \ \hbox{in}\ \MM_{\rm hom}\ \ \ (j=1,2)\ ,\]
    where $A_i$ is an abelian variety with $\dim A_i=h^{2,1}(X_i)=5$.
    Likewise, since $H^2(Y_i)=N^1 H^2(Y_i)$, we have
    \[ h(Y_i)= \bigoplus_j \LL(n_j)\ \ \ \hbox{in}\ \MM_{\rm hom}\ \ \ (j=1,2)\ .\]    
  
  The K\"unneth components of the diagonal of $X$ are algebraic (as $X$ is a hypersurface). The only interesting cohomology of $X$ is $H^6(X)$. From the remarks above, we obtain
  a homomorphism
    \[ \begin{split} \Gamma_6\colon\ \ h^6(X)\ \to\ h^6(\wt{X_1\times X_2})&= h^3(X_1)\otimes h^3(X_2)\oplus \bigoplus_j \LL(n_j)\\
           &\ \to\ h^2(A_1\times A_2)(2)\oplus \bigoplus_j \LL(n_j)\ \ \ \hbox{in}\ \MM_{\rm hom} \ .\\
           \end{split}\]
    Since both $X$ and the $A_i$ verify the Lefschetz standard conjecture, and $\Gamma_6$ induces an injection on cohomology, $\Gamma_6$ has a left--inverse (by semi--simplicity of $\MM_{\rm num}$, as in the proof of proposition \ref{ab}). It follows there is a homomorphism
    \[ \Gamma\colon\ \ h(X)\ \to\ h^2(A_1\times A_2)(2)\oplus \bigoplus_j \LL(n_j)\ \ \ \hbox{in}\ \MM_{\rm hom} \]
    admitting a left--inverse. Using finite--dimensionality of the motive of $X$ (and lemma \ref{split}), we find the homomorphism
    \[ \Gamma\colon\ \ h(X)\ \to\ h^2(A_1\times A_2)(2)\oplus \bigoplus_j \LL(n_j)\ \ \ \hbox{in}\ \MM_{\rm rat} \]
    also admits a left--inverse 
   
   We have seen there is a homomorphism
   \[ h(F)\ \to\ h(X\times X)(-2)\oplus \bigoplus_{i=1}^{n-1}  h(X)(i-2)\ \ \ \hbox{in}\ \MM_{\rm rat}\ \]
  admitting a left--inverse. Combining with the split homomorphism $\Gamma$, this proves proposition \ref{ab6}. 
\end{proof}

We proceed to wrap up the proof of proposition \ref{propvoe6}. It follows from proposition \ref{ab6} (after taking Chow groups) there is a correspondence $\Gamma$ inducing a split injection
    \[ \Gamma_\ast\colon\ \ A^j_{num}(F)\ \to\ (\Pi_4^{A\times A})_\ast A^{j-2}_{num}(A\times A)\oplus \bigoplus_j (\Pi_2^{A})_\ast A^{n_j}_{num}(A)\ ,\] 
   and the left inverse to $\Gamma_\ast$ is again induced by a correspondence. 
    This implies there is a split injection
    \[ \Gamma_\ast\colon\ \ A^j_{num}(F)\ \to\  A^{j-2}_{(2j-8)}(A\times A)\oplus \bigoplus_j  A^{2}_{(2)}(A)\ ,\] 
    where $A^i_{(j)}$ denotes the Beauville filtration \cite{Beau}. (Indeed, one has $(\Pi_2^{A})_\ast A^{1}_{num}(A)=0$ and     $(\Pi_2^{A})_\ast A^{i}_{}(A)=A^i_{(2i-2)}(A)=0$ for $i\ge 3$ \cite{Beau}.)    
    
    The summand $A^{2}_{(2)}(A)$ consists of smash--nilpotent cycles: indeed, thanks to K\"unnemann's hard Lefschetz theorem \cite{Kun}, there is an isomorphism
    \[ A^2_{(2)}(A)\ \xrightarrow{\cong}\ A^g_{(2)}(A)\]
    (where $g:=\dim A$),
    which is induced by a correspondence and such that the inverse is again induced by a correspondence. It remains to consider the first summand. It follows from general properties of Beauville's filtration \cite{Beau} that
    \[ A^{j-2}_{(2j-8)}(A\times A) =0\ \ \ \hbox{for\ }j\le 3\ \hbox{and\ for\ }j\ge 7\ .\]
    In the remaining cases ($j=5$ or $6$), we have
    \[      A^{j-2}_{(2j-8)}(A\times A)=\begin{cases} A^3_{(2)}(A\times A) &  \hbox{if\ }j=5\ ;\\
                                         A^4_{(4)}(A\times A) & \hbox{if\ } j=6\ .\\
                                         \end{cases} \]
    In both cases, it follows that 
      \[  A^{j-2}_{(2j-8)}(A\times A)\ \subset\ A^{j-2}_\otimes(A\times A)\ .\]
      Indeed, in the first case ($j=5$), there is a correspondence--induced isomorphism
         \[    A^3_{(2)}(A\times A)\oplus \bigoplus_j  A^{2}_{(2)}(A)\ \xrightarrow{\cong}\  A^{2g-1}_{(2)}(A\times A)  \ ,\]   
  ( where $g:=\dim A$), thanks to K\"unnemann's hard Lefschetz theorem \cite{Kun}. As the inverse to this isomorphism is also correspondence--induced, this isomorphism preserves
  smash--nilpotence. But the right--hand side consists of smash--nilpotent cycles, by Sebastian's result \cite{Seb2}.
  In the second case (i.e., $j=6$), we have a (correspondence--induced) isomorphism (with correspondence--induced inverse)
    \[ A^4_{(4)}(A\times A)\ \xrightarrow{\cong}\  A^{2g}_{(4)}(A\times A)\ .\]
    But the right--hand side obviously consists of smash--nilpotent cycles, and so we are done.
   \end{proof}

\begin{remark} The argument of proposition \ref{propvoe6} runs into problems for $j=4$. This is because for $j=4$, one would need to know that 
  \[A^2_{num}(A\times A)\cap A^2_{(0)}(A\times A)=0\ ,\] 
  which is one of the unsolved cases of Beauville's conjectures.
\end{remark}

\section{Arbitrary dimension}

In this section, we consider injectivity of certain cycle class maps for Fano varieties of lines on cubics of arbitrary dimension. In the special case of Fermat cubics, we obtain an optimal statement (proposition \ref{fermat}). For general cubics of arbitrary dimension, we obtain a weaker statement (proposition \ref{propDM}).

\begin{proposition}\label{fermat} Let $X\subset\PP^{n+1}(\C)$ be the Fermat cubic 
   \[ x_0^3+x_1^3+\cdots +x_{n+1}^3 =0\ .\]
   Let $F=F(X)$ be the Fano variety of lines on $X$ (so $F$ is smooth of dimension $2n-4$).
   Then
   \[ A^j_{hom}(F)=0\ \ \ \hbox{for\ all\ }j\ge {5n-2\over 3}\ .\]
   \end{proposition}
   
   \begin{proof} As before, it follows from \cite{moi} that for any $j$ there is a split injection
   \begin{equation}\label{inj}  \begin{split} A^j_{hom}(F)\ \to\ A^{j+2}_{hom}(X^{[2]})\ &\to\ A^{j+2}_{hom}(X\times X)\oplus \bigoplus_{i=1}^{n-1} A^{j+2-i}_{hom}(X)\\
                                                                          &=\ A^{j+2}_{hom}(X\times X)\oplus \bigoplus_{i=-1}^{n-3} A^{j-i}_{hom}(X)\ .\\
                                                                          \end{split}\end{equation}
                                                                          
                                        We now use the following result concerning Chow groups of Fermat hypersurfaces:
                                        
                        \begin{theorem}[Voisin \cite{V96}]\label{chowfermat} Let $X\subset\PP^{n+1}(\C)$ be the Fermat cubic. Then
                        \[   A^j_{hom}(X)=0\ \ \ \hbox{for\ all\ } j\ge {2n+4\over 3}\ .\]
                        \end{theorem}
                        
                      \begin{proof} This follows from a more general result \cite{V96}, or alternatively \cite[ Theorem 4.4]{Vo}. (Voisin's result actually concerns a more general family of cubics.)
                      \end{proof}

                 \begin{corollary}\label{prod} Let $X\subset\PP^{n+1}(\C)$ be the Fermat cubic. Then
                 \[ A^j_{hom}(X\times X)=0\ \ \ \hbox{for\ all\ }j\ge {5n+4\over 3}\ .\]
                 \end{corollary}
                 
   \begin{proof} This follows from theorem \ref{chowfermat}, combined with the following observation: if $X$ and $X^\prime$ are two smooth projective varieties and
     \[ A_i^{hom}(X)=A_i^{hom}(X^\prime)=0\ \ \ \hbox{for\ all\ } i\le r\ ,\]
     then also
     \[ A_i^{hom}(X\times X^\prime)=0\ \ \ \hbox{for\ all\ } i\le r\ .\]
     The observation can be proven using the Bloch--Srinivas argument \cite{BS}, cf. for instance \cite[Remark 1.8.2]{small}.
     \end{proof}              
   
   Using proposition \ref{chowfermat} and corollary \ref{prod}, we verify that the right--hand side of the map (\ref{inj}) is $0$ as soon as $j\ge (5n-2)/3$. Since the map (\ref{inj}) is injective, we are done.
   \end{proof}

\begin{remark} Conjecturally, theorem \ref{chowfermat} is true for {\em all\/} smooth cubic hypersurfaces; this would follow from the Bloch--Beilinson conjectures.
\end{remark}

The following conjecture is a special case of a more general conjecture (the conjecture as stated in \cite[Conjecture 6.4]{DMan} concerns Fano varieties of linear subspaces (not only lines) contained in arbitrary complete intersections (not only cubics)):

\begin{conjecture}[Debarre--Manivel \cite{DMan}]\label{DMconj} Let $X\subset\PP^{n+1}(\C)$ be a smooth cubic, and let $F=F(X)$ be the Fano variety of lines on $X$. Assume $\ell\in\NN$ is such that
  $ {(\ell+4)(\ell+3)/ 2}\le n+1$.
Then the inclusion of $F$ into the Grassmannian $G:=G(1,\PP^{n+1})$ induces an isomorphism
  \[ A_\ell(F)\ \cong\ A_\ell(G)\ ; \]
  in particular $A_\ell^{hom}(F)=0$.
  \end{conjecture}
  
Conjecture \ref{DMconj} is true for $\ell\le 1$ \cite[Section 6]{DMan}. We provide some partial answer for higher $\ell$:

\begin{proposition}\label{propDM} Let $X\subset\PP^{n+1}(\C)$ be a smooth cubic, and let $F=F(X)$ be the Fano variety of lines on $X$.

\noindent
(\rom1) 
  \[ A_2^{hom}(F)=0\ \ \ \hbox{for\ all\ }n\ge 14\ .\]
  
 \noindent
 (\rom2)
 \[ A_3^{hom}(F)=0\ \ \ \hbox{for\ all\ }n\ge 18\ .\]
 \end{proposition} 

\begin{proof}

\noindent
(\rom1) Suppose $n\ge 14$. It follows from \cite[Corollary 1]{Otw} that 
   \[ A_\ell^{hom}(X)=0\ \ \ \hbox{for\ all\ }\ell\le 4\ .\]
   Using the argument of corollary \ref{prod}, we find that
   \[  A_\ell^{hom}(X\times X)=0\ \ \ \hbox{for\ all\ }\ell\le 4\ .\]
  It follows from theorem \ref{main} that there are split injections
   \[ A^j_{hom}(F)\ \to\ A^{j+2}_{hom}(X^{[2]})\ \to\  A^{j+2}_{hom}(X\times X)\oplus \bigoplus_{i=-1}^{n-3} A^{j-i}_{hom}(X)\ .\]
   From what we have just said, we know that the right--hand side is zero as soon as $j+2\ge 2n-4$, i.e. $j\ge \dim F-2$.
   
  \noindent
  (\rom2) Suppose $n\ge 18$. Then we have that
   \[ A_\ell^{hom}(X)=0\ \ \ \hbox{for\ all\ }\ell\le 5\ \] 
   \cite[Corollary 1]{Otw}. Again, this implies that also
        \[  A_\ell^{hom}(X\times X)=0\ \ \ \hbox{for\ all\ }\ell\le 5\ .\]
        There is a split injection
\[ A^j_{hom}(F)\ \to\  A^{j+2}_{hom}(X\times X)\oplus \bigoplus_{i=-1}^{n-3} A^{j-i}_{hom}(X)\ .\]
The right--hand side is zero as soon as $j+2\ge 2n-5$, i.e. $j\ge \dim F-3$.
\end{proof}

\vskip1cm
\begin{nonumberingt} The first seeds of this note were planted in the course of the Strasbourg 2014---2015 groupe de travail based on the monograph \cite{Vo}. Thanks to all the participants of this groupe de travail for the fertile atmosphere, and for frequently watering my seeds. 
Many thanks to Yasuyo, Kai and Len for lots of pleasant coffee breaks.
\end{nonumberingt}

\end{document}